\documentclass[12pt]{article}

\usepackage{amsmath}
\usepackage{amssymb}

 \newcommand{\beq}{\begin{equation}}
\newcommand{\eeq}{\end{equation}}

\numberwithin{equation}{section}

\newtheorem{theorem}{Theorem}[section]
\newtheorem{proposition}{Proposition}[section]

\newtheorem{corollary}[theorem]{Corollary}

\newenvironment{proof}{\medbreak\noindent{\it Proof.}\rm}{\hfill$\square$\rm}

\newcommand{\PSH}{{\operatorname{PSH}}}

\newcommand{\Vol}{{\operatorname{Vol}}}
\newcommand{\Covol}{{\operatorname{Covol}}}
\newcommand{\Capa}{{\operatorname{Cap}}}
\newcommand{\Log}{{\operatorname{Log}\,}}
\newcommand{\Exp}{{\operatorname{Exp}\,}}
\newcommand{\BE}{{\mathbf E}}
\newcommand{\D}{{\mathbb D}}

\newcommand{\R}{{\mathbb R}}
\newcommand{\Rn}{{\mathbb R}^n}
\newcommand{\Rnm}{{\mathbb R}_-^n}
\newcommand{\Rnp}{{\mathbb R}_+^n}
\newcommand{\C}{{\mathbb  C}}
\newcommand{\ii}{{\mathbb  I}}

\newcommand{\Cn}{{\mathbb  C\sp n}}
\newcommand{\F}{{\mathcal F}}

\newcommand{\cL}{{\mathcal  L}}

\newcommand{\vph}{\varphi}

\begin{document}

\begin{center}
{\Large\bf Copolar convexity}
\end{center}

\begin{center}
{\large Alexander Rashkovskii}
\end{center}

\vskip1cm

\begin{abstract}
We introduce a new operation, {\it copolar addition}, on unbounded convex subsets of the positive orthant of $\Rn$ and establish convexity of the covolumes of the corresponding convex combinations. The proof is based on a technique of geodesics of plurisubharmonic functions. As an application, we show that there are no relative extremal functions inside a non-constant geodesic curve between two toric relative extremal functions.

\medskip\noindent
{\sl Mathematic Subject Classification}: 32U15, 32U20, 52A20, 52A39
\end{abstract}

\section{Introduction}
Recall that volumes of Minkowski convex combinations $P_t:=(1-t)P_0+t\,P_1$ of two convex bodies $P_j\subset \Rn$  satisfy
the classical Brunn-Minkowski inequality
$$\Vol(P_t)^{\frac1n} \ge (1-t)\,\Vol(P_0)^{\frac1n}+t\,\Vol(P_1)^{\frac1n}.$$
 For its various aspects and consequences, see \cite{Gar}.
We would like to mention here a few variants of the inequality. In the first one (proved under certain symmetry conditions),  the Minkowski combinations $P_t$ are replaced by certain {log-Minkowski combinations} \cite{BLYZ}, \cite{Sa}, \cite{XL}.
Another version considers, instead of the combinations $P_t$, balls in the norms of the complex Calder\'{o}n interpolated spaces \cite{CE}. Finally,
in \cite{KT}, convex bodies are replaced by {\it coconvex} ones, that is, bounded subsets of a closed strictly convex cone $C\subset\Rn$ whose complements to $C$ are convex, the operation on the coconvex set being taking the complement to the Minkowski sum of their complements. Consideration of such sets is motivated by their relations (mostly be means of Newton polyhedra \cite{Ku1}) to singularity theory, commutative algebra, and complex analysis \cite{KaK}, \cite{KV}, \cite{R3}, \cite{R4}, \cite{R9}.
In such a setting,  a {\sl reversed} Brunn-Minkowski inequality for coconvex sets was established in \cite{KT}; in terms of their complements $P_t$, it can be written as \beq\label{rBM}\Covol(P_t)^{\frac1n} \le (1-t)\,\Covol(P_0)^{\frac1n}+t\,\Covol(P_1)^{\frac1n},\eeq
where
\beq\label{covol} \Covol(P)=\Vol(C\setminus P).\eeq

In this note, we will deal (for the sake of simplicity) with proper convex subsets of the positive orthant $\Rnp$ of $\Rn$ with bounded complements, and the Minkowski addition will be replaced with a different operation, $\oplus$, which we will call {\it copolar addition} (see  (\ref{cop}) and (\ref{newsum}) for the definition), also motivated by consideration of Newton polyhedra and thus relevant to the aforementioned application areas. In a sense, such an operation is more natural than the Minkowski addition; for example, the copolar sum of cosimplices (complements of simplices) is still a cosimplex, while their Minkowski sum is a more complicated polyhedron. In Theorem~\ref{theo}, we prove the inequality
\beq\label{main}\Covol(P_t^\oplus)\le (1-t)\,\Covol(P_0)+ t\,\Covol(P_1),\eeq
which is strict unless $P_0=P_1$. We give a simple example (see a remark after the statement of Theorem~\ref{theo}) for which (\ref{main}) is sharper than (\ref{rBM}).

We prove (\ref{main}) by developing recent results on geodesics of plurisubharmonic functions \cite{R16} for relative extremal functions in a toric setting. In particular, we prove that, as in the case of compact manifolds, the Legendre transform of the convex image of any toric geodesic $u_t$ on a bounded Reinhardt domain is an affine function of $t$ (Theorem~\ref{guan1}).

Note that (\ref{main}) can be equivalently described as {\sl convexity} of the Monge-Amp\'ere capacity of multiplicative combinations $K_t^\times$ of logarithmically convex Reinhardt compacts of $\Cn$. Such combinations appear in the toric case of the complex Calder\'{o}n interpolated spaces, the corresponding version of the Brunn-Minkowski inequality \cite{CE}, \cite{CEK} expresses however {\sl logarithmic concavity} of the volumes of $K_t^\times$ (for convex $K_j$).

 The results are applied to plurisubharmonic functions and their singularities. In particular, we show that there is no relative extremal function inside a non-constant geodesic curve between two toric relative extremal functions (Corollary~\ref{uniq}).

\section{Newton sets and Monge-Amp\'ere masses}
\label{Motiv}
Here we recall some results on Monge-Amp\'ere masses of plurisubharmonic functions, related to the notion of Newton polyhedron. They can be viewed as a development of the fact that the multiplicity of a generic holomorphic mapping at the origin of $\Cn$ equals $n!$ times the covolume of its Newton polyhedron \cite{Ku1}.

\bigskip
\noindent
{\bf 2.1.}
 The first set of results concerns characteristics of singularities of psh functions, obtained in \cite{LeR},  \cite{R3}, \cite{R4}, \cite{R9}. Let $\vph$ be a plurisubharonic function on a neighborhood of $0\in\Cn$ such that $\vph(0)=-\infty$, and let $\Psi_\vph$ be its {\it indicator}, i.e.,
\beq\label{indic} \Psi_\vph(z)=\limsup_{y\to z} \lim_{m\to \infty} m^{-1}\vph(y_1^m,\ldots,y_n^m).\eeq
It is known that $\Psi_\vph$ is a toric (that is, independent of the arguments of the variables) nonpositive plurisubharmonic function in the unit polydisk $\D^n$, satisfying
$\Psi_\vph(|z_1|^c,\ldots,|z_n|^c)=c\,\Psi_\vph(z)$ for any $c>0$, and
$\vph\le \Psi_\vph+O(1)$ near $0$. If the Monge-Amp\'ere current $(dd^c \vph)^n$ is well defined near $0$ (for example, this is so if $0$ is an isolated singularity of $\vph$), then the Monge-Amp\'ere measure of $\Psi_\vph$ is concentrated at $0$:
\beq\label{Psi}(dd^c\Psi_\vph)^n=N_\vph\,\delta_0\eeq
 for some $N_\vph\ge 0$ (the {\it Newton number} of $\vph$ at $0$), and $(dd^c \vph)^n(0)\ge N_\vph$.

The Newton number can be computed as follows. Denote $$\Rnp=\{a\in\Rn:\: a_j>0,\ 1\le j\le n\},\quad \Rnm=-\Rnp,$$
and let the mappings $ \Exp: \Rnm\to\D^n$ and $\Log: \D^n\cap (\C_*)^n\to \Rnm$ be defined as
$$\Exp s = (e^{s_1},\ldots,e^{s_n}),\quad \Log z=(\log|z_1|,\ldots,\log|z_n|).$$ The {\it convex image} $\psi_\vph=\Exp^*\Psi_\vph$ of the indicator $\Psi_\vph$ is a convex, positive homogeneous function on $\Rnm$, i.e., $\psi_\vph(c\,s)=c\,\psi_\vph(s)$ for any $c>0$ and all $s\in\Rnm$. Therefore, $\psi_\vph$ is the restriction of the support function of a convex set $\Gamma_\vph\subset\Rnp$ to $\Rnm$:
$$\psi_\vph(s)=\sup_{a\in\Gamma_\vph} \langle a,s\rangle,\quad s\in\Rnm$$
(note that the support function of $\Gamma_u $ equals $+\infty$ outside $\overline{\Rnm}$). By relating the complex and real Monge-Amp\'ere operators, one computes then
\beq\label{Nvol} N_\vph=n!\,\Covol( \Gamma_\vph),\eeq
where the covolume is defined by (\ref{covol}) with $C=\Rnp$.
Moreover, the generalized Lelong number (in the sense of Demailly) $\nu(u,\vph):=dd^cu\wedge (dd^c\Psi_\vph)^{n-1}(0)$ of any psh function $u$ near $0$ with respect to $\Psi_\vph$ can be represented as a mixed covolume of the corresponding convex sets and computed as
$$\nu(u,\vph)=n!\int_{L_\vph} |\psi_\vph(s)|\,d\gamma_\vph(s),$$
where $\gamma_\vph$ is a positive measure supported on the set of extreme points of the convex set
\beq\label{Lvph} L_\vph=\{s\in\Rnm:\: \psi_\vph(s)\le -1\}.\eeq

For a holomorphic mapping $F$, the set $\Gamma_{\log|F|}$ is the {\sl Newton polyhedron of $F$ at $0$} in the sense of \cite{Ku1}, that is, the convex hull of the set
$$\bigcup_{k\in A} (k+\Rnp),$$ where $A$ is the collection of all multi-indices $k$ such that $z^k$ has a nonzero coefficient in the Taylor expansion of at least one of the components of $F$.

\bigskip
\noindent
{\bf 2.2.}   A slightly different setting studied in \cite{ARZ} concerns relative Monge-Amp\'ere capacities $\Capa(K)$
of subsets $K$ of bounded hyperconvex domains $\Omega\subset\Cn$ .
We recall that $\Capa(K)$ equals the total Monge-Amp\'ere mass of the relative extremal function
\beq\label{relext}\omega_K(z)=\limsup_{y\to z}\,\sup\{u(y):\: u\in\PSH(\Omega),\ u<0, \ u|_K\le-1\};\eeq
namely,
$$ \Capa(K)=(dd^c\omega_K)^n(\Omega)=(dd^c\omega_K)^n(K).$$

Let $\Omega$ be the unit polydisk $\D^n$, $D\Subset \D^n$ be a nonempty complete logarithmically convex Reinhardt domain, and $L=L_K\subset\Rnm$ be the logarithmic image of $K=\overline D$, that is, $L=\Log(K\cap (\C_*)^n)$. Then
\beq\label{omega} \omega_K(z)=\sup_{a\in\Rnp}\frac{\sum a_k\log|z_k|}{|h_D(a)|}, \quad z\in \D^n\setminus D,\eeq
and
\beq\label{CapK}\Capa(K)=n!\,\Covol(\Gamma_L),\eeq
where
\beq\label{GL} \Gamma_L=\{a\in\Rnp:\: h_L(a)\le -1\}\eeq
and $h_L(a)$ is the restriction of the support function of the (convex) set $L$ to $\Rnp$:
$$h_L(a)=\sup_{s\in L} \langle a,s\rangle,\quad s\in\Rnp$$
(and $h_L$ equals $+\infty$ outside $\overline\Rnp$).
Note that the right hand side of (\ref{omega}) can be written as $\Log^* h_{\Gamma_L}$.

In both the examples, one deals with pairs $\{L,\Gamma\}$ of unbounded convex subsets of $\Rnp$ and $\Rnm$, and the Monge-Amp\'ere masses are computed as covolumes of the corresponding subsets of $\Rnp$. The difference here is that in the first setting, one starts with a set in $\Rnp$ and arrives to a set in $\Rnm$, while in the second case the things go the other way round.

\section{Copolars}
\label{Copolars}

Now let us for a moment strip away the plurisubharmonic content of the above settings and concentrate on its convex counterpart.

Let $C_+$ denote the collection of closed convex subsets $\Gamma$ of $\overline\Rnp$ that are complete in the sense that $\Gamma+\Rnp\subset\Gamma$. Similarly, we introduce the family $C_-=-C_+$ of closed convex complete subsets of $\overline\Rnm$.

Given a set $A\in C_+$ or $A\in C_-$, we introduce its {\it copolar} $A^\circ$ as
\beq\label{cop} A^\circ=\{x\in\Rn:\: \langle x,y\rangle\le -1\ \forall y\in A\}.\eeq
Equivalently, the copolars can be described by means of the {\it support functions}
$$ h_A(x)=\sup_{y\in A} \langle x,y\rangle,\quad x\in\Rn;$$
namely,
$ A^\circ=\{x\in\Rn:\: h_A(x)\le -1\}$.

\begin{proposition}
\begin{enumerate}
\item If $A\in C_+$, then $A^\circ\subset\overline\Rnm$;
\item if $A\in C_-$, then $A^\circ\subset\overline\Rnp$;
\item $(C_+)^\circ=C_-$ and $ (C_-)^\circ=C_+$;
\item if $A$ belongs either to $C_+$ or to $C_-$, then $(A^{\circ})^\circ=A$.
\end{enumerate}
\end{proposition}

\begin{proof} The first two assertions follow from the completeness of $A$, and the third one follows from the convexity.

To prove the last one, assume $A\in C_-$ be closed, consider the support functions $h_A$ and $h_{A^\circ}$ and put
$B=\{y\in\Rnm:\: h_{A^\circ}(y)\le -1\}$. We claim that $B=A$.

First, let $a\in A$, then $h_A(x)\ge \langle x,a\rangle$ for any $x\in \Rnp$. Denote
$$\delta A^\circ={\overline{\partial L^\circ\cap\Rnp}}=\{x\in\Rn:\: h_A(x)= -1\}.$$
Then
$$ h_{A^\circ}(a)=\sup\{\langle x,a\rangle:\: x\in \delta A^\circ\}\le \sup\{h_A(x):\: x\in \delta A^\circ\}=-1, $$
which shows $A\subset B$.

To prove the reverse inclusion, take any $a\in \Rnm\setminus A$. Since the set $A$ is complete, there exists a point $x^\ast\in \delta A^\circ$ such that
$ \langle x^\ast,a\rangle > h_L(a^\ast)$. We have then
$$ h_{A^\circ}(a)=\sup\{\langle x,a\rangle:\: x\in \delta A^\circ\}\ge \langle x^\ast,a\rangle > h_L(a^\ast)=-1,$$
so $a\not\in B$ and thus, $B\subset A$.
\end{proof}

\medskip

So far, we have had no essential difference between the collections $C_+$ and $C_-$. From now on, we will be considering a subclass of sets in $C_+$ whose copolars will not belong to the corresponding subclass of $C_-$. Namely, we will restrict ourselves to the class $CC_+$ of {\it cobounded} sets in $C_+$ (those with bounded complements to $\Rnp$). It is easy to see that $P\in CC_+$  if and only if $P^\circ\subset s^\ast+\Rnm $ for some $s^\ast\in\Rnm$; we denote the latter subclass of $C_-$ by $UC_-$, so $(CC_+)^\circ= UC_-$ and $(UC_-)^\circ=CC_+$.

\medskip

Let $\cL\{f\}$ denote the {\it Legendre transform} of a function $f$:
$$ \cL[f](y)=\sup_{x\in\Rn}\{\langle x,y\rangle -f(x)\}.$$
We recall that $\cL[f]$ is always a convex, lower semicontinuous ({\it lsc}) function and $\cL[\cL[f]]$ is the largest convex lsc minorant of $f$.

In particular, the support function $h_A$ of a closed convex set $A$ is the Legendre transform of the indicator function $\ii_A$ of $A$ (which equals $0$ on $A$ and $+\infty$ on $\Rn\setminus A$), and the following formula (which is to be used later on) for cutoff functions expresses the copolarity as a Legendre duality.

\begin{proposition}\label{Leg} Let $L\in UC_-$, then
$$ \cL[\max \{h_{L^\circ},-1\}]=\max\{h_L+1,0\}.$$
\end{proposition}

\begin{proof} As is known,
$$\cL[\min\{f,g\}]=\max\{\cL[f],\cL[g]\}. $$
Applying this to $f=\ii_{L^\circ}$ and $g=\ii_{\Rnp}+1$, we conclude that that $\cL[\max \{h_{L_0},-1\}]$ is the largest convex lsc minorant of the function equal to $0$ on $L^\circ$ and to $1$ on $\Rnp\setminus L^\circ$. In particular, its restriction to any segment $l$ connecting the origin with abitrary $p\in\delta L^\circ$ does not exceed the affine function $F_l$ on $l$ defined by $F_l(0)=1$ and $F_l(p)=0$.

Since  the function $h_L+1$ satisfies the bounds and its restriction to $l$ coincides with $F_l$ for any $l$, this completes the proof.
\end{proof}

\medskip

We introduce the {\it copolar addition} on $CC_+$ as
\beq\label{newsum} P\oplus Q=\left(P^\circ + Q^\circ\right)^\circ.\eeq
Evidently, $P\oplus Q\in CC_+$ if $P,Q\in CC_+$.

\medskip

{\bf Example.} A set $L\in CC_+$ of the form $L=\{a\in\Rnp:\: \langle a,b_L\rangle\ge 1\}$ for some vector $b_L\in\Rnp$ will be called a {\it cosimplex}. It is easy to see that if $P$ and $Q$ are cosimplices in $\Rnp$, then $P\oplus Q$ is a cosimplex as well, and its reference vector $b=b_{P\oplus Q}$  is the one whose components $b_j$ satisfy
$$b_j^{-1}=b_{Pj}^{-1} + b_{Qj}^{-1},\quad 1\le j\le n.$$

\section{Convexity of covolumes}

Given $P_0,P_1\in CC_+$, we form the collection of {\it copolar combinations}
\beq\label{copcom} P_t^\oplus=\left((1-t)P_0^\circ+tP_1^\circ\right)^\circ,\quad 0< t< 1.\eeq

\begin{theorem}
\label{theo}
Copolar combinations of sets $P_0,P_1\in CC_+\setminus\{\overline\Rnp\}$ satisfy
\beq\label{CoBM}\Covol(P_t^\oplus)\le (1-t)\,\Covol(P_0)+ t\,\Covol(P_1),\quad 0<t<1;\eeq
an equality in (\ref{CoBM}) occurs for some $t\in (0,1)$ if and only if $P_0=P_1$.
\end{theorem}

{\it Remark.} Unlike standard versions of Brunn-Minkowski type inequalities, there is no $\frac1n$ exponents in (\ref{CoBM}).
As was shown in \cite{KT} (in a more general setting), the covolumes of sets from $CC_+$ satisfy a reversed Aleksandrov-Fenchel inequality and, as a consequence, the reversed Brunn-Minkowski inequality (\ref{rBM}) with respect to the standard Minkowski combinations
$P_t=(1-t)P_0+t\,P_1$. None of the inclusions $ P_t^\oplus\subset P_t$ and $P_t\subset P_t^\oplus$ is generally true. As an example, consider
$$P_0=\{a\in \R^2:\: {a_1}/3+a_2\ge 1\}, \quad P_1=\{a\in \R^2:\: a_1+{a_2}/3\ge 1\}.$$ Then $P_{1/2}$ is the polyhedron in $CC_+$ with vertices at $(2,0)$, $(1/2,1/2)$ and $(0,2)$, while $P_{1/2}^\oplus=\{a\in \R^2:\: a_1+a_2\ge 3/2\}$. Since $\Covol(P_{1/2})=1$, $\Covol(P_{1/2}^\oplus)=9/8$ and $\Covol(P_0)=\Covol(P_1)$, in this example (\ref{CoBM}) is sharper than (\ref{rBM}).

We do not know if a corresponding inequality for the copolar addition involving $\Covol^{\frac1n}$ is true.

\begin{proof}
To prove (\ref{CoBM}), we recall the motivation examples from Section~\ref{Motiv} and run them the opposite way. Let $P\in CC_+\setminus\{\overline\Rnp\}$, then $L:=P^\circ\in (CC_+)^\circ$. In the terminology of Section~\ref{Copolars}, the set $\Gamma_L$ defined by (\ref{GL}) is exactly the copolar $L^\circ$ to $L$, that is, $\Gamma_L=P$.
The set $K:=\overline{\Exp(L)}$
is the closure of a complete logarithmically convex Reinhardt subdomain of the unit polydisk $\D^n$. By (\ref{CapK}), we have then
\beq\label{CapK1} \Capa(K) =n!\, \Covol(P).\eeq

Starting with $P_0,P_1\in CC_+\setminus\{\overline\Rnp\}$, we get the family $P_t^\oplus\subset CC_+$ of their copolar combinations and then, as described above, the corresponding Reinhardt compact subsets $K_t^\times$ of $\D^n$ that are actually the multiplicative combinations of $K_0$ and $K_1$:
 $$\Log K_t^\times= (1-t)\Log K_0+t\,\Log K_1.$$
 Now we use \cite[Thm. 4.3]{R16} stating that, in this setting,
 \beq\label{capin}\Capa(K^\times)\le (1-t)\,\Capa(K_0)+t\,\Capa(K_1).\eeq
By (\ref{CapK1}) and the definition of copolar addition, we get the desired inequality.

The equality statement will be proved in Section~\ref{Appl} (see Corollary~\ref{uniq1}).
\end{proof}

\medskip
{\it Remark.} The multiplicative combinations $K_t^\times$ of convex Reinhardt bodies is the toric version of the balls in the norms of the complex Calder\'{o}n interpolated spaces, mentioned in the introduction. As was proved in \cite{CE}, their {\sl volumes} satisfy the ({\sl non-reversed}) Brunn-Minkowski inequality
\beq\label{CE} \Vol(K_t^\times) \ge \Vol(K_0)^{1-t}\,\Vol(K_1)^{t},\eeq
which is a bound on the size of $K_t^\times$ opposite to (\ref{capin}). Note also that (\ref{capin}) means {\sl convexity} of the capacities for {\sl logarithmically convex} Reinhardt bodies, while (\ref{CE}) expresses  {\sl logarithmic concavity} of the volumes for {\sl convex} Reinhardt bodies.

\medskip
In the next section, we explain where the crucial inequality (\ref{capin}) comes from.

\section{Representation of toric geodesics}

The proof of Theorem~\ref{theo} is based on a technique of geodesics, developed first for metrics on compact K\"{a}hler manifolds (see \cite{G12} and bibliography therein) and applied then to plurisubharmonic functions on domains of $\Cn$ in \cite{R16}, \cite{Ho}. In the latter setting, plurisubharmonic functions on a bounded hyperconvex domain $\Omega\subset\Cn$ with zero bounded values on $\partial\Omega$ are considered. A {\it subgeodesic} of two such functions, $u_0$ and $u_1$, is a family of functions $v_t$, $0<t<1$, such that $v(z,\zeta):=v_{\log|\zeta|}(z)$ belongs to the class $W(u_0,u_1)$ of nonpositive plurisubharmonic functions on the product $\Omega\times\{1<|\zeta|<e\}$ whose boundary values on $\Omega\times\{\log|\zeta|=j\}$ do not exceed $u_j$. The {\it geodesic} of $u_0$ and $u_1$ is their largest subgeodesic.

It was shown in \cite{R16} that for $u_j$ from Cegrell's energy class $\F_1(\Omega)$, the geodesic $u_t$ attains $u_j$ as its boundary values (as the uniform limits if $u_j$ are bounded, and in capacity in the general case) as $t\to j$ and, furthermore, the energy functional
$$\BE(u)= \int_\Omega u(dd^c u)^n$$ is affine on $u_t$. If $\omega_K$ is the relative extremal function (\ref{relext}) of a compact subset $K$ of $\Omega$, then $$\BE(\omega_K)=-\Capa (K),$$ and this is where the Monge-Amp\'ere capacities come into the picture.

For toric plurisubharmonic functions, this can be translated to the language of convex functions on $\Rnm$ (as was done in \cite{BB} for a global setting of $\Rn$), and the geodesics give rise to the multiplicative convex combinations $K_t^\times$ as in the proof of Theorem~\ref{theo}. Moreover, such geodesics can be described in a way which is a bit similar to the copolar addition. Any toric plurisubharmonic function $u$ on $\D^n$ with zero boundary values on $\partial\D^n$ can be identified with its convex image
$$\check u=\Exp^* u,$$ which is a convex function on $\Rnm$, increasing in each variable and equal to $0$ on $\partial\Rnp$.
Let $u_0$ and $u_1$ be two toric plurisubharmonic functions on $\D^n$ with zero boundary values,
 and let $u_t$, $0<t<1$, be the corresponding geodesic. Geodesics on the space of K\"{a}hler metrics on toric K\"{a}hler manifolds can be characterized as those whose Legendre transforms are affine in $t$ (\cite{Gu}, see also \cite{BB}, \cite{G15}), which can be proved by differentiating the defining equation for the Legendre transform of the geodesic. A similar fact is true in the local setting as well; since we have a non-smooth situation, we give an independent proof.

\begin{theorem}\label{guan1}  Let $u_t$ be the geodesic of two toric plurisubharmonic function $u_0$ and $u_1$ on $\D^n$ with zero boundary values. Then its convex image $\check u_t$ has the representation
 \beq\label{guan}\check u_t =\cL\left[ (1-t)\cL[\check u_0] + t\cL[\check u_1]\right].
 \eeq
 \end{theorem}

{\it Remark.} Observe a resemblance between (\ref{guan}) and  (\ref{copcom}).

\begin{proof}
Note that $\check u(t,s):=\check u_t(s)$ is convex in $n+1$ variables, although $\cL[\check u]$ is not. To avoid this difficulty, we consider the $(n+1)$-dimensional Legendre transform
$$\cL^*[f](y,r)=\sup\{\langle x,y\rangle + t\cdot r -f(x,t):\: x\in\Rn,\ r\in\R\}$$
of convex functions $f(x,t)$ on $\Rn\times\R$ and apply it to the function $\check u$; to get $\check u$ defined on the whole $\Rn\times\R$, we set it equal to $+\infty$ when $s\in\Rn\setminus \Rnm$ or $t\in\R\setminus [0,1]$. We will still use the $\cL[u_t]$ denotation for the Legendre transform of $u_t(s)$ in $s$.

First let us assume the functions $u_j$ to be bounded: $u_j\ge -M_j$.
Take any bounded toric subgeodesic $v_t$ connecting $u_0$ and $u_1$. Replacing it, if necessary, with $\max\{v_t, u_0-M_1\,t,u_1-M_0\,(1-t)\}$, we can assume that $v_t$ converges uniformly to $ u_j$ as $t\to j\in\{0,1\}$.
It follows from the definition of the Legendre transform that
$$\cL^*[\check v](a,r)=\sup_{0\le t\le 1}\{\cL[\check v_t](a)+t\cdot r\}$$
for the convex image $\check v$ of $v_t$. Therefore, we have
$$\cL^*[\check v](a,r)\ge (\cL[\check v_t](a)+t\cdot r)|_{t=0}=\cL[\check u_0](a)$$
and
$$\cL^*[\check v](a,r)\ge (\cL[\check v_t](a)+t\cdot r)|_{t=1}=\cL[\check u_1](a)+r,$$
so
\beq\label{checkL}\cL^*[\check v](a,r)\ge \max\{\cL[\check u_0](a), \cL[\check u_1](a)+r\}.\eeq

Let now $\check w_t(s)$ be defined by the right hand side of (\ref{guan}). Evidently, it is convex in $(s,t)$. Since $\cL[\check w_t]\to \cL[\check u_j]$ uniformly as $t\to j$, its plurisubharmonic image $w\in W(u_0,u_1)$. Furthermore,
\begin{eqnarray*}\cL^*[\check w](a,r) &=& \sup_{0\le t\le 1}\{(1-t)\cL[\check u_0](a)+t\cL[\check u_1]+t\cdot r\}\\
&=& \sup_{0\le t\le 1}\{t\left[r-\cL[\check u_0](a)+\cL[\check u_1](a)\right]+\cL[u_0]\}\\
&=& \max \{\cL[\check u_0](a), \cL[\check u_1](a)+r\}.
\end{eqnarray*}
Comparing this with (\ref{checkL}), we conclude $\cL^*[\check w]\le \cL^*[\check v]$ and so, $\check w\ge \check v$ for any subgeodesic $v_t$. Therefore, $w_t$ is the geodesic.

In the general case of unbounded $u_j$, we apply what has already been proved to the cutoff functions $u_{j,N}=\max\{u_j,-N\}$. The convex images of their geodesics are $$\check u_{t,N} =\cL\left[ (1-t)\cL[\check u_{0,N}] + t\cL[\check u_{1,N}]\right],$$
and letting $N\to\infty$, we get (\ref{guan}).
\end{proof}

\medskip
{\it Remark.} The same proof works for toric plurisubharmonic functions on arbitrary bounded hyperconvex Reinhardt domains.

\section{Applications}\label{Appl}

Here we apply the results to geodesics of toric relative extremal functions and to toric singularities. Using this, the uniqueness part of Theorem~\ref{theo} is proved.

\bigskip
\noindent
{\bf 6.1.} As was indicated in Section~\ref{Motiv}.2, the relative extremal function $\omega_K$ can be expressed in terms of the support function of the polar $L^\circ$ (denoted there by $\Gamma_L$) of the logarithmic image $L=\Log K$ of $K$. Namely,
\beq\label{omegal0} \omega_K(z)=\max \{h_{L^\circ}(\Log\, z),-1\}, \quad z\in \D^n.\eeq

 Now take $K_0,K_1\Subset\D^n$ that are the closures of two complete logarithmically convex Reinhardt domains, and let $u_j=\omega_{K_j}$. By (\ref{omegal0}),
 $$ \check u_j =\max\{h_{L_j^\circ},-1\},\quad j=0,1,$$
 and thus, by Proposition~\ref{Leg},
 $$\cL[\check u_j] =\max\{h_{L_j}+1,0\},\quad j=0,1.$$
Therefore, (\ref{guan}) gives us the following explicit formula for computing the geodesics of relative extremal functions.

\begin{theorem}\label{formula}
Let $K_j\subset\D^n$ be the closures of nonempty complete logarithmically convex Reinhardt domains $D_j$, $j=0,1$, then the geodesic $u_t$  of $u_j=\omega_{K_j}$ represents as
$$ \check u_t=\cL[(1-t)\max\{h_{L_0}+1,0\} + t \max\{h_{L_1}+1,0\}],\quad 0<t<1,$$
where $L_j\in UC_-$ are the logarithmic images of $K_j$.
\end{theorem}

\begin{corollary}\label{uniq} In the conditions of Theorem~\ref{formula}, $u_t$ is the relative extremal function for some $t\in (0,1)$ if and only if $K_0=K_1$.
\end{corollary}

\begin{proof} Assume $u_t=\omega_K$ for some $t\in (0,1)$, and let $L=\Log K$. Then
\beq\label{cor} (1-t)\max\{h_{L_0}(a)+1,0\} + t \max\{h_{L_1}(a)+1,0\}=\max\{h_L(a)+1,0\},\ a\in\Rnp.\eeq
In particular, for all $a$ close to the origin, $h_L(a)=(1-t)h_{L_0}(a)+ t h_{L_1}(a)$, which implies
\beq\label{sum} L=(1-t)L_0 + t L_1.\eeq

On the other hand, (\ref{cor}) for $a$ close to $\delta L^\circ=\{x\in\Rn:\: h_A(x)= -1\}$ implies $L^\circ =L_0^\circ\cap L_1^\circ$, which contradicts (\ref{sum}) unless $L_0=L_1$.
\end{proof}

\begin{corollary}\label{uniq1} The equality statement of Theorem~\ref{theo} is true.
\end{corollary}

\begin{proof} Since an equality in (\ref{capin}) occurs if and only if the geodesic $u_t$ is the relative extremal function of $K_t^\times$, the statement follows from Corollary~\ref{uniq}.
\end{proof}

\bigskip
\noindent
{\bf 6.2.} This can be applied to investigation of residual Monge-Amp\'ere masses of toric plurisubharmonic singularities.
As was mentioned in Section~\ref{Motiv}.1, the support function $h_P$ of a set $P\in CC_+$ is the convex image of an indicator $\Psi_P\in\PSH(\D^n)$, and the set $L_\vph\subset\Rnm$ defined by (\ref{Lvph}) is the copolar to $P$. The copolar addition on $CC_+$ induces the copolar addition of the indicators:
$$ \Psi_P\oplus\Psi_Q= \Psi_{P\oplus Q};$$
furthermore, it gives rise to the copolar combinations
\beq\label{copind}\Psi_t^\oplus= \Psi_{P_t^\oplus}\eeq
of $\Psi_{P_j}$ for $P_0,P_1\in CC_+$, where the sets
$ P_t^\oplus$ are defined by (\ref{copcom}).

For example, cosimpleces $P_j=\{a\in\Rnp:\: \langle a,b^j\rangle\ge 1\}$ generate the indicators
$$\Psi_{P_j}(z)=\Phi_{b^j}(z):=\max_{1\le k\le n}\frac{\log|z_k|}{b_k^j},\quad j=0,1,$$
so in this case
$$\Psi_t^\oplus=\Phi_{c(t)} ,$$
where the components $c_k$ of the weight vector $c(t)$ are defined by
$$\frac1{c_k}=\frac{1-t}{b_k^0}+\frac{t}{b_k^1},\quad 1\le k\le n.$$

By (\ref{Psi}) and (\ref{Nvol}),
the Monge-Amp\'ere mass $(dd^c\Psi_t^\oplus)^n(0)$ equals $n!\,\Covol(P_t^\oplus)$, so Theorem~\ref{theo} and Corollary~\ref{uniq} imply

\begin{corollary} Let $\Psi_j$ be indicators corresponding to complete Reinhardt sets $K_j={\overline D_j}\subset\D^n$, $j=0,1$. Then the Monge-Amp\'ere measures of their copolar combinations (\ref{copind}) satisfy
$$ (dd^c\Psi_t^\oplus)^n\le (1-t)(dd^c\Psi_0)^n + t\,(dd^c\Psi_1)^n,\quad 0<t<1,$$
and the inequality is strict unless $\Psi_0=\Psi_1$.
\end{corollary}

{\it Remark.} As follows from \cite[Thm. 6.2]{R16}, no pair $\Psi_P,\Psi_Q$ of {\sl different} indicators can be connected by a geodesic because it does not exceed $\Psi_{P\cap Q}\le \min\{\Psi_P,\Psi_Q\}$ (this can also be easily deduced from Theorem~\ref{guan1}). Moreover, the main result of \cite{Ho} implies that two toric plurisubharmonic functions $\vph_0,\vph_1$ with zero boundary values and isolated singularities at $0$ can be connected by a geodesic $\vph_t$ (in the sense that $\vph_t\to \vph_j$ in capacity as $t\to j\in\{0,1\}$) if and only if their indicators $\Psi_{\vph_j}$ (\ref{indic}) are equal.

\bigskip
{\small {\bf Acknowledgement.}
The author thanks Genki Hosono for useful conversations.}

\vskip1cm

Tek/Nat, University of Stavanger, 4036 Stavanger, Norway

\vskip0.1cm

{\sc E-mail}: alexander.rashkovskii@uis.no

\end{document}